\documentclass{article}
\usepackage[]{graphicx}
\usepackage[]{amsmath,amssymb}
\usepackage[]{amsthm,enumerate}
\usepackage{verbatim,bm}
\usepackage{color}
\usepackage{xcolor}

\newtheorem{theorem}{Theorem}[section]
\newtheorem{lemma}[theorem]{Lemma}

\theoremstyle{definition}

\newtheorem{example}[theorem]{Example}

\theoremstyle{remark}
\newtheorem{remark}[theorem]{Remark}
\newtheorem{problem}[theorem]{Problem}

\def\j{{\bm j}}   
\def\0{{\bm 0}}   

\begin{document}

\title{A family of balanced generalized weighing matrices}
\author{
 Hadi Kharaghani\thanks{Department of Mathematics and Computer Science, University of Lethbridge,
Lethbridge, Alberta, T1K 3M4, Canada. \texttt{kharaghani@uleth.ca}}
\and
Thomas Pender\thanks{Department of Mathematics and Computer Science, University of Lethbridge,
Lethbridge, Alberta, T1K 3M4, Canada. \texttt{Thomas.pender@uleth.ca}}
\and
  Sho Suda\thanks{Department of Mathematics,  National Defense Academy of Japan, Yokosuka, Kanagawa 239-8686, Japan. \texttt{ssuda@nda.ac.jp}}
}
\date{\today}

\maketitle

\begin{abstract}
Balanced weighing matrices with parameters $$
\left(1+18\cdot\frac{9^{m+1}-1}{8},9^{m+1},4\cdot 9^m\right),
$$
for each nonzero integer $m$ are constructed. This is the first infinite class not belonging to those with classical parameters.  
It is shown that any balanced  weighing matrix   is equivalent to a five-class association scheme.

\end{abstract}

\section{Introduction}

A \emph{weighing matrix} of order $v$ and weight $k$ is a $(0,\pm 1)$-matrix $W$ of order $v$ such that $WW^T = kI$. The special cases in which $k=v-1$ or $k=v$ yield the \emph{Conference} and \emph{Hadamard} matrices, respectively. For this paper, we are interested in those weighing matrices that are \emph{balanced}, that is, those weighing matrices which, upon setting each non-zero entry to unity, {yield} the incidence matrix of a symmetric $(v,k,\lambda)$ balanced incomplete block design. 

Balanced  weighing matrices with {\it classical parameters} include those with parameters $(v=\frac{q^{n+1}-1}{q-1},k=q^n,\lambda=q^{n}-q^{n-1}),$ for each positive integer $n$. Besides the balanced weighing matrices with these parameters, a known sporadic example is one with parameters $(19,9,4)$, constructed by Mathon and appeared first in \cite{ds83} and then it was shown in \cite{GM} to be the only balanced weighing matrix with the mentioned parameters. 

Symmetric designs are shown to have the structure of a three-class association scheme, completely determined by its parameters \cite[Theorem~1.6.1]{BCN}. One of the fundamental problems in association schemes is to find larger schemes whose quotient or fission  schemes coincide with the original one.

It is the purpose of this paper to extend this single sporadic case to an infinite class of balanced  weighing matrices and to show that any balanced weighing matrix corresponds to a five-class association scheme, where a three-class association scheme corresponding to a symmetric design is quotient of a fusion scheme of the five-class scheme, and vice versa.



\section{Preliminaries}
\subsection{BGWs over $\mathbb{Z}_n$}
Let $G$ be a multiplicatively written finite group, and let $W=(w_{ij})_{i,j=1}^v$ be a (0,G)-matrix. We say that $W$ is a \emph{balanced generalized weighing matrix} over $G$ with parameters $(v,k,\lambda)$, denoted BGW$(v,k,\lambda;G)$, if each row of $W$ contains exactly $k$ nonzero entries, and if the multiset $\{w_{ik}w_{jk}^{-1}\mid 1\leq k\leq v, w_{ik} \neq 0 \neq w_{j,k}\}$ contains exactly $\lambda/|G|$ copies of every element of $G$, for any distinct $i,j\in\{1,\ldots,v\}$. 

\begin{example}
A BGW$(v,k,\lambda;\{1,-1\})$ is a balanced weighing matrix of order $v$ and weight $k$.
\end{example}

\begin{example}
Every conference matrix of order $v$ is a BGW$(v,v-1,v-2;\{1,-1\})$. 
\end{example}

\begin{example}
It is known that (see \cite{GM})
for any prime power $q$, there is a BGW$(\frac{q^{m+1}-1}{q-1},q^m,q^{m-1}(q-1))$ over the cyclic group $G$ whose order divides $q-1$. This is true for each positive integer $m$, hence there is a balanced weighing matrix with these parameters for any such $q$.
\end{example}

To be concrete we will let our group be $G=\{1,w,\ldots,w^{n-1}\}$, the cyclic group of order $n$ defined by $w=e^{2\pi\sqrt{-1}/n}$, i.e. the complex $n^\text{th}$ roots of unity.  A balanced generalized weighing matrix over $G\cup\{0\}$ is called \emph{balanced Butson}.
Let $W$ be a balanced Butson weighing matrix over $G\cup\{0\}$ with parameters $(v,k,\lambda)$. 
Since $W$ is a matrix over $G\cup\{0\}$, there are $n$ disjoint (0,1)-matrices $\{W_i\}_{i=0}^{n-1}$ satisfying
\begin{align*}
W&=\sum_{i=0}^{n-1} w^i W_i.
\end{align*}
Moreover, we have the following necessary condition on $\{W_i\}_{i=0}^{n-1}$, namely,
\begin{align*}
\sum_{i,j=0}^{n-1} w_i w_j^{-1}W_i W_j^{\top}=\sum_{i,j=1}^n w_j^{-1}w_iW_j^{\top}W_i =k I_{v}+\frac{\lambda}{n}G(J_v-I_v), 
\end{align*}    
as matrices with entries of the group ring $\mathbb{C}[G]$, where $I_v$ is the identity matrix of order $v$ and $J_d$  denotes the $v\times v$ all-one matrix.  
Comparing entries with each element in $G$ yields the following lemma. 
\begin{lemma}\label{lem:bgw}
\begin{align*}
\sum_{i=0}^{n-1} W_i W_i^{\top}&=\sum_{i=0}^{n-1} W_i^{\top}W_i =k I_v+\frac{\lambda}{n}(J_v-I_v),\\
\sum_{i=0}^{n-1} W_i W_{i+j}^{\top}&=\sum_{i=0}^{n-1} W_{i+j}^{\top}W_i =\frac{\lambda}{n}(J_v-I_v), 
\end{align*}
where $j\in\{1,\ldots,n-1\}$ and the indices are taken modulo $n$. 
\end{lemma}

\subsection{Orthogonal arrays}


Let $S = \{1,2,\dots,q\}$ be some finite alphabet. An \emph{orthogonal array} of strength $t$ and index $\lambda$ is an $N \times k$ matrix over $S$ such that in every $N \times t$ subarray, each $t$-tuple in $S^t$ appears $\lambda$ times. We denote this property as OA$_\lambda(N,k,q,t)$. 

For $t=2e$, the following lower bound on $N$ was shown by Rao (see \cite[Theorem~2.1]{HSS}), namely, $N\geq \sum_{i=0}^e \binom{k}{i}(q-1)^i$. An orthogonal array with parameters $(N,k,q,2e)$ is said to be complete if the equality holds in above. 


\begin{lemma}\label{lem:oa}
Let $q$ be an odd prime power, and let $n \geq 2$ be an integer. 
There exists an orthogonal array $A$ with parameters $(q^n,\frac{q^n-1}{q-1},q,2)$ of index $q^{n-2}$ such that $A=\sum_{i=1}^{q} i A_i$, where $A_i$ ($i\in\{1,\ldots,q\}$) are disjoint $q^n \times \frac{q^n-1}{q-1}$ $(0,1)$-matrices satisfying
\begin{enumerate}
\item $\sum_{i=1}^{q} A_iA_i^\top=\frac{q^{n-1}-1}{q-1} J_{q^n}+q^{n-1} I_{q^n}$,
\item $\sum_{i,j=1,i\neq j}^{q}A_iA_j^\top=q^{n-1}(J_{q^n}-I_{q^n})$.
\end{enumerate}
\end{lemma}
\begin{proof}
It is shown (see \cite[Theorem~3.20]{HSS}) that for $q$ an odd prime power and a positive integer $n\geq 2$, there exists an orthogonal array $A$ with parameters $(q^n,\frac{q^n-1}{q-1},q,2)$. 
Note that $A$ is complete.   
Then by \cite[Theorem~5.21]{D}, the Hamming distances between distinct rows take only one value, say $d$. 
This value is uniquely determined as follows.  
Let $K_{k,q,i}(x)$ be the Krawtchouk polynomial of degree $i$ defined as
$
K_{k,q,i}(x)=\sum_{j=0}^i(-1)^j (q-1)^{i-j}\binom{x}{j}\binom{k-x}{i-j}.
$ 
Then the Hamming distance $d$ satisfies 
$$
K_{\frac{q^n-1}{q-1},q,0}(d)+K_{\frac{q^n-1}{q-1},q,1}(d)=0,
$$
which shows that $d=\frac{q^{n-1}-1}{q-1}$. 
\\
(i): Let $x,y\in\{1,\ldots,q^n\}$. Write the $x$-th row of $A$ as $a_x$.  Then 
\begin{align*}
(\sum_{i=1}^q A_i A_i^\top)_{x,y}&=\sum_{i=1}^q (A_i A_i^\top)_{x,y}\\\displaybreak[0]
&=\sum_{i=1}^q \sum_{\ell=1}^{\frac{q^n-1}{q-1}}(A_i)_{x,\ell}(A_i)_{y,\ell}\\\displaybreak[0]
&=\sum_{i=1}^q |\{\ell\in\{1,\ldots,\frac{q^n-1}{q-1}\} \mid A_{x,\ell}=A_{y,\ell}=i\}|\\
&= |\{\ell\in\{1,\ldots,\frac{q^n-1}{q-1}\} \mid A_{x,\ell}=A_{y,\ell}\}|\\
&=\frac{q^n-1}{q-1}-d(a_x,a_y)
\end{align*}
which shows that $\sum_{i=1}^q A_i A_i^\top=\frac{q^{n-1}-1}{q-1} J_{q^n}+q^{n-1} I_{q^n}$. \\
(ii): Since $\sum_{i=1}^q A_i=J_{q^n,\frac{q^n-1}{q-1}}$, we have $\sum_{i,j=1}^q A_i A_j^\top=(\sum_{i=1}^q A_i)(\sum_{j=1}^q A_j^\top)=J_{q^n,\frac{q^n-1}{q-1}}J_{\frac{q^n-1}{q-1},q^n}=\frac{q^n-1}{q-1} J_{q^n}$. This with (i) shows  (ii). 
\end{proof}

\section{Construction}
Consider the following matrices,
\[
\arraycolsep=1.0pt\def\arraystretch{1.0}
U=
\left[\begin{array}{cccccccccccccccccc}
0&0&0&0&0&0&0&0&0&1&1&1&1&1&1&1&1&1\\
-&0&0&1&0&1&1&1&0&0&0&0&1&0&-&1&-&0\\
0&-&0&1&1&0&0&1&1&0&0&0&-&1&0&0&1&-\\
0&0&-&0&1&1&1&0&1&0&0&0&0&-&1&-&0&1\\
1&1&0&-&0&0&1&0&1&1&-&0&0&0&0&1&0&-\\
0&1&1&0&-&0&1&1&0&0&1&-&0&0&0&-&1&0\\
1&0&1&0&0&-&0&1&1&-&0&1&0&0&0&0&-&1\\
1&0&1&1&1&0&-&0&0&1&0&-&1&-&0&0&0&0\\
1&1&0&0&1&1&0&-&0&-&1&0&0&1&-&0&0&0\\
0&1&1&1&0&1&0&0&-&0&-&1&-&0&1&0&0&0\\
\end{array}
\right]
\]
{and}
\[
\arraycolsep=1.0pt\def\arraystretch{1.0}
V=\left[
\begin{array}{cccccccccccccccccc}
1&1&1&1&1&1&1&1&1&0&0&0&0&0&0&0&0&0\\
0&-&-&0&1&0&0&0&1&-&1&1&0&1&0&0&0&1\\
-&0&-&0&0&1&1&0&0&1&-&1&0&0&1&1&0&0\\
-&-&0&1&0&0&0&1&0&1&1&-&1&0&0&0&1&0\\
0&0&1&0&-&-&0&1&0&0&0&1&-&1&1&0&1&0\\
1&0&0&-&0&-&0&0&1&1&0&0&1&-&1&0&0&1\\
0&1&0&-&-&0&1&0&0&0&1&0&1&1&-&1&0&0\\
0&1&0&0&0&1&0&-&-&0&1&0&0&0&1&-&1&1\\
0&0&1&1&0&0&-&0&-&0&0&1&1&0&0&1&-&1\\
1&0&0&0&1&0&-&-&0&1&0&0&0&1&0&1&1&-\\
\end{array}
\right],
\]
where $-$ stands for $-1$.  
 Both matrices $U$ and $V$ are signed residual designs of a symmetric $BIBD(19,9,4)$. Moreover, if $\bar{Z}$ is the matrix obtained from the matrix $Z$ by taking the absolute value of each entry, then it can be seen that $\bar{U}$ and $\bar{V}$ are complementary incidence matrices.

The derived design associated with both $\bar{U}$ and $\bar{V}$ can also be signed as well, and it is the matrix $Y$  below.
 \[
\arraycolsep=1.0pt\def\arraystretch{1.0}
Y=\left[
\begin{array}{cccccccccccccccccc}
 {0} & 0 & 0 & 1 & 0 & - & 1 & - & 0 & 0 & - & - & 0 & 1 & 0 & 0 & 0 & 1 \\
 0 & 0 & 0 & - & 1 & 0 & 0 & 1 & - & - & 0 & - & 0 & 0 & 1 & 1 & 0 & 0 \\
 0 & 0 & 0 & 0 & - & 1 & - & 0 & 1 & - & - & 0 & 1 & 0 & 0 & 0 & 1 & 0 \\
 1 & - & 0 & 0 & 0 & 0 & 1 & 0 & - & 0 & 0 & 1 & 0 & - & - & 0 & 1 & 0 \\
 0 & 1 & - & 0 & 0 & 0 & - & 1 & 0 & 1 & 0 & 0 & - & 0 & - & 0 & 0 & 1 \\
 - & 0 & 1 & 0 & 0 & 0 & 0 & - & 1 & 0 & 1 & 0 & - & - & 0 & 1 & 0 & 0 \\
 1 & 0 & - & 1 & - & 0 & 0 & 0 & 0 & 0 & 1 & 0 & 0 & 0 & 1 & 0 & - & - \\
 - & 1 & 0 & 0 & 1 & - & 0 & 0 & 0 & 0 & 0 & 1 & 1 & 0 & 0 & - & 0 & - \\
 0 & - & 1 & - & 0 & 1 & 0 & 0 & 0 & 1 & 0 & 0 & 0 & 1 & 0 & - & - & 0 
\end{array}
\right]
\]
Each of $U$ and $Y$, and $V$ and $Y$ can be used in turn to give two BGW$(19,9,4;\{1,-1\})$ matrices.\\

The fact that these signings work together in this way is remarkable, and for the interested reader, we pose the following question: 
\begin{problem}
 Is there a BGW$(2q^2+1, q^2, (q^2-1)/2; \{1,-1\})$, for every odd prime power $q$?
 \end{problem}
The method used in this paper can be extended to any BGW$(2q^2+1, q^2, (q^2-1)/2; \{1,-1\})$ satisfying properties similar to the BGW$(19,9,4;\{1,-1\})$ matrix above, for every odd prime power $q$.''

The matrices $U,V,$ and $Y$  satisfy the following fundamental properties.   

\begin{lemma}\label{lem:pqy}
\begin{enumerate}
\item $UU^\top=V^\top=9I_{10}$, $UV^\top=VU^\top$.
\item $YY^\top =9I_9-J_9$.
\item $UY^\top=VY^\top=0$. 
\item $\overline{U}\cdot\overline{U}^\top=\overline{V}\cdot\overline{V}^\top=5I_{10}+4J_{10}$, $\overline{U}\cdot\overline{V}^\top=\overline{V}\cdot\overline{U}^\top=-5I_{10}+5J_{10}$.
\item $\overline{Y}\cdot\overline{Y}^\top =5I_9+3J_9$, 
$\overline{U}\cdot\overline{Y}^\top=\overline{V}\cdot\overline{Y}^\top=4J_{10}$. 
\end{enumerate}
\end{lemma}
\begin{proof}
A straightforward calculation.
\end{proof}

Let $W$ be a BGW$(\frac{9^{m+1}-1}{8},9^m,9^m-9^{m-1};\mathbb{Z}_4)$ over $\mathbb{Z}_4$ and define the \\$10\frac{9^{m+1}-1}{8}\times 18\frac{9^{m+1}-1}{8}$ matrix $R$ by 
$$
R=W_0\otimes U+W_1\otimes (-V)+W_2\otimes (-U)+W_3\otimes V.
$$

Let $A$ be the OA$(9^{m+1},\frac{9^{m+1}-1}{8},9,2)$ of index $9^{m-1}$ 
 over $\{1,\ldots,9\}$. Write $A=\sum_{i=1}^9 iA_i$, where the $A_i$'s are disjoint $(0,1)$-matrices.
We then define the $9^{m+1} \times 18\frac{9^{m+1}-1}{8}$ matrix $D$ by $D=\sum_{i=1}^9 A_i\otimes r_i$.
\begin{lemma}
\begin{enumerate}
\item $RR^\top=9^{m+1} I_{\frac{10(9^{m+1}-1)}{8}}$.
\item $DD^\top=9^{m+1} I_{9^{m+1}}-J_{9^{m+1}}$.
\item $RD^\top=DR^\top=0$. 
\item $\overline{R}\cdot \overline{R}^\top=5\cdot 9^m I_{\frac{10(9^{m+1}-1)}{8}}+4\cdot 9^mJ_{\frac{10(9^{m+1}-1)}{8}}$.
\item $\overline{D}\cdot\overline{D}^\top=5\cdot 9^m I_{9^{m+1}}+(4\cdot 9^{m} -1)J_{9^{m+1}}$.
\item $\overline{R}\cdot\overline{D}^\top=\overline{D}\cdot\overline{R}^\top=4\cdot 9^{m+1}J_{\frac{10(9^{m+1}-1)}{8},9^{m+1}}$, where $J_{a,b}$ denotes the $a\times b$ all-one matrix.
\end{enumerate}
\end{lemma}
\begin{proof}
(i): By Lemma~\ref{lem:bgw} and Lemma~\ref{lem:pqy} (i),  
\begin{align*}
RR^\top
&=9W_0W_0^\top \otimes I_{10}-W_0W_1^\top\otimes UV^\top-9W_0W_2^\top\otimes I_{10}+W_0W_3^\top\otimes UV^\top\\
&-W_1W_0^\top \otimes VU^\top+9W_1W_1^\top\otimes I_{10}+W_1W_2^\top\otimes VU^\top-9W_1W_3^\top\otimes I_{10}\\
&-9W_2W_0^\top \otimes I_{10}+W_2W_1^\top\otimes UV^\top+9W_2W_2^\top\otimes I_{10}-W_2W_3^\top\otimes UV^\top\\
&+W_3W_0^\top \otimes VU^\top-9W_3W_1^\top\otimes I_{10}-W_3W_2^\top\otimes VU^\top+9W_3W_3^\top\otimes I_{10}\\
&=9(\sum_{i=0}^3W_iW_i^\top-\sum_{i=0}^3W_iW_{i+2}^\top)\otimes I_{10}
-(\sum_{i=0}^3W_iW_{i+1}^\top -\sum_{i=0}^3W_iW_{i+3}^\top)\otimes UV^\top\\
&=9^{m+1} I_{\frac{10(9^{m+1}-1)}{8}}.
\end{align*}

(ii): By Lemma~\ref{lem:oa} and Lemma~\ref{lem:pqy} (ii),  
\begin{align*}
DD^\top&=\sum_{i,j=1}^{9} A_i A_j^\top \otimes r_i r_j^\top\\\displaybreak[0]
&=\sum_{i=1}^{9} A_i A_i^\top \otimes r_i r_i^\top+\sum_{i\neq j} A_i A_j^\top \otimes r_i r_j^\top\\\displaybreak[0]
&=8\sum_{i=1}^{9} A_i A_i^\top -\sum_{i\neq j} A_i A_j^\top \\\displaybreak[0]
&=(9^{m}-1) J_{9^{m+1}}+8\cdot 9^{m} I_{9^{m+1}}- 9^{m}(J_{9^{m+1}}-I_{9^{m+1}})\\\displaybreak[0]
&=9^{m+1} I_{9^{m+1}}-J_{9^{m+1}}. 
\end{align*}

(iii): By Lemma~\ref{lem:pqy} (iii),  
\begin{align*}
RD^\top&=(W_0\otimes U+W_1\otimes (-V)+W_2\otimes (-U)+W_3\otimes V)(\sum_{i=1}^9 A_i^\top\otimes r_i^\top)\\
&=\sum_{i=1}^9(W_0A_i^\top\otimes Ur_i^\top+ W_1A_i^\top\otimes (-V)r_i^\top+ W_2A_i^\top\otimes (-U)r_i^\top+ W_3A_i^\top\otimes Vr_i^\top)\\
&=0. 
\end{align*}

(iv): Note that $\overline{R}=W_0\otimes \overline{U}+W_1\otimes \overline{V}+W_2\otimes \overline{U}+W_3\otimes \overline{V}$. The claim is proved in a similar fashion as (i) by Lemma~\ref{lem:bgw} and Lemma~\ref{lem:pqy} (iv).   

(v): Note that $\overline{A}=\sum_{i=1}^9 A_i\otimes \overline{r_i}$. The claim is proved in a similar fashion as (ii) by Lemma~\ref{lem:bgw} and Lemma~\ref{lem:pqy} (v).  

(vi): The claim is proved in a similar fashion as (iii), but we include its proof. 
By Lemma~\ref{lem:pqy} (vi), the fact that $\sum_{i=1}^9A_i=J_{9^{m+1},\frac{9^{m+1}-1}{8}}$ and that $\sum_{i=0}^3 W_i$ is a symmetric design with parameters $(\frac{9^{m+1}-1}{8},9^m,9^m-9^{m-1})$,
\begin{align*}
\overline{R}\cdot\overline{D}^\top&=(W_0\otimes \overline{U}+W_1\otimes \overline{V}+W_2\otimes \overline{U}+W_3\otimes \overline{V})(\sum_{i=1}^9 A_i^\top\otimes \overline{r_i}^\top)\\\displaybreak[0]
&=\sum_{i=1}^9(W_0A_i^\top\otimes \overline{U}\overline{r_i}^\top+ W_1A_i^\top\otimes \overline{V}\overline{r_i}^\top+ W_2A_i^\top\otimes \overline{U}\overline{r_i}^\top+ W_3A_i^\top\otimes \overline{V}\overline{r_i}^\top)\\
&=4\sum_{i=1}^9(W_0A_i^\top+ W_1A_i^\top+ W_2A_i^\top+ W_3A_i^\top)\\
&=4(W_0+W_1+W_2+W_3)(\sum_{i=1}^9A_i^\top)\\
&=4(W_0+W_1+W_2+W_3)J_{\frac{9^{m+1}-1}{8},9^{m+1}}\\
&=4\cdot 9^m J_{\frac{9^{m+1}-1}{8},9^{m+1}}. \qedhere
\end{align*}
\end{proof}
    
We claim that $X=\begin{bmatrix} \j & D \\ \0 & R \end{bmatrix}$, where $\j,\0$ are the all-one and zero column  vectors of appropriate order respectively, is a BGW$(18\frac{9^{m+1}-1}{8}+1,9^{m+1},4\cdot 9^{m};\mathbb{Z}_2)$. 
Indeed, 
\begin{align*}
XX^\top&=\begin{bmatrix} \j & D \\ \0 & R \end{bmatrix}\begin{bmatrix} \j^\top & \0^\top \\ D^\top & R^\top \end{bmatrix}=\begin{bmatrix} \j \j^\top+DD^\top & D R^\top \\ RD^\top & RR^\top \\
\end{bmatrix}\\
&=\begin{bmatrix} J_{9^{m+1}}+(9^{m+1} I_{9^{m+1}}-J_{9^{m+1}}) & 0 \\ 0 & 9^{m+1} I_{\frac{10(9^{m+1}-1)}{8}} \\
\end{bmatrix}\\
&=9^{m+1}I_{18\frac{9^{m+1}-1}{8}+1},
\end{align*}
and 
\begin{align*}
\overline{X}\cdot \overline{X}^\top&=\begin{bmatrix} \j & \overline{D} \\ \0 & \overline{R} \end{bmatrix}\begin{bmatrix} \j^\top & \0^\top \\ \overline{D}^\top & \overline{R}^\top \end{bmatrix}=\begin{bmatrix} \j \j^\top+\overline{D}\cdot\overline{D}^\top & \overline{D}\cdot\overline{R}^\top \\ \overline{R}\cdot\overline{D}^\top & \overline{R}\cdot\overline{R}^\top \\
\end{bmatrix}\\
&=\begin{bmatrix} J_{9^{m+1}}+(5\cdot 9^m I_{9^{m+1}}+(4\cdot 9^{m} -1)J_{9^{m+1}}) & 4\cdot 9^{m+1}J_{\frac{10(9^{m+1}-1)}{8},9^{m+1}} \\ 4\cdot 9^{m+1}J_{9^{m+1},\frac{10(9^{m+1}-1)}{8}} & 5\cdot 9^m I_{\frac{10(9^{m+1}-1)}{8}}+4\cdot 9^mJ_{\frac{10(9^{m+1}-1)}{8}} \\
\end{bmatrix}\\
&=5\cdot 9^{m}I_{18\frac{9^{m+1}-1}{8}+1}+4\cdot 9^{m}J_{18\frac{9^{m+1}-1}{8}+1}. 
\end{align*}

Thus we have shown the following theorem. 
\begin{theorem}
Let $m$ be any positive integer. There exists a balanced weighing matrix with parameters 
$$
\left(1+18\cdot\frac{9^{m+1}-1}{8},9^{m+1},4\cdot 9^m\right).
$$
\end{theorem}

\begin{remark}
There has been a deliberate use of balanced Butson matrices instead of the general balanced weighing matrices in the proof. The presented new method makes minimal use of balancedness and may help with other similar constructions. 
\end{remark}

\section{Balanced  weighing matrices   and association schemes}
In this section we show an equivalence between BGW matrices over $\{1,-1\}$ and some association schemes. 

Let $n$ be a positive integer. 
Let $X$ be a finite set and $R_i$ ($i\in\{0,1,\ldots,n\}$) be a nonempty subset of $X\times X$. 
The \emph{adjacency matrix} $A_i$ of the graph with vertex set $X$ and edge set $R_i$ is a $(0,1)$-matrix indexed by $X$ such that $(A_i)_{xy}=1$ if $(x,y)\in R_i$ and $(A_i)_{xy}=0$ otherwise. 
A \emph{{\rm(}symmetric{\rm)} association scheme} with $n$ classes is a pair $(X,\{R_i\}_{i=0}^n)$ satisfying the following:
\begin{enumerate}[(\text{AS}1)]
\item $A_0=I_{|X|}$.
\item $\sum_{i=0}^n A_i = J_{|X|}$.
\item $A_i^\top =A_i$ for any $i\in\{1,\ldots,n\}$.
\item For any $i,j$, and $k$, there exists an integer $p_{i,j}^k$ such that $A_i A_j=\sum_{k=0}^n p_{i,j}^k A_k$.
\end{enumerate}
We will also refer to $(0,1)$-matrices $A_0,A_1,\ldots,A_n$ satisfying (AS1)-(AS4) as an association scheme. 
The vector space spanned by $A_i$'s forms a commutative algebra, denoted by $\mathcal{A}$ and called the \emph{Bose-Mesner algebra}.
There exists a basis of $\mathcal{A}$ consisting of primitive idempotents, say $E_0=(1/|X|)J_{|X|},E_1,\ldots,E_n$. 
Note that $E_i$ is the projection onto a maximal common eigenspace of $A_0,A_1,\ldots,A_n$. 
Since  $\{A_0,A_1,\ldots,A_n\}$ and $\{E_0,E_1,\ldots,E_n\}$ are two bases in $\mathcal{A}$, there exist the change-of-bases matrices $P=(p_{ij})_{i,j=0}^n$, $Q=(q_{ij})_{i,j=0}^n$ so that
\begin{align*}
A_j=\sum_{i=0}^n p_{ij}E_i,\quad E_j=\frac{1}{|X|}\sum_{i=0}^n q_{ij}A_i.
\end{align*}
The matrices $P=(p_{ij})_{i,j=0}^d$ and $Q=(q_{ij})_{i,j=0}^d$ are the {\it first and second eigenmatrices} respectively. 

Let $(X,\{R_i\}_{i=0}^n)$, $(X,\{R'_i\}_{i=0}^{n'})$ be association schemes. 
If there exists a partition $\Lambda_0:=\{0\},\Lambda_1,\ldots,\Lambda_{n'}$ of $\{0,1,\ldots,n\}$ such that $R'_i=\cup_{j\in\Lambda_i}R_j$ for any $i\in\{0,1,\ldots,n'\}$, then $(X,\{R_i\}_{i=0}^n)$ is said to be {\it fission} of $(X,\{R'_i\}_{i=0}^{n'})$ and $(X,\{R_i'\}_{i=0}^{n'})$ is said to be {\it fusion} of $(X,\{R_i\}_{i=0}^{n})$. 

The scheme is \emph{imprimitive} if,
on viewing the $A_i$ as adjacency matrices
of graphs $G_i$ on vertex set $X$, at least one of the $G_i$, $i \ne 0$, is disconnected. 
In this case, there exists a set $\mathcal{I}$ of indices such that $0$ and such $i$ are elements of $\mathcal{I}$ and $\sum_{j\in\mathcal{I}}A_j=I_p\otimes J_q$ for some $p,q$ with $p>1$ after permuting the vertices suitably. 
Thus the set $X$ is partitioned into $p$ subsets called \emph{fibers}, each of which has size $q$. 
The set $\mathcal{I}$ defines an equivalence relation on $\{0,1,\ldots,d\}$ by $j\sim k$ if and only if $p_{i,j}^k\neq 0$ for some $i\in \mathcal{I}$.  
Let $\mathcal{I}_0=\mathcal{I},\mathcal{I}_1,\ldots,\mathcal{I}_t$ be the equivalence classes on $\{0,1,\ldots,d\}$ by $\sim$. 
Then by \cite[Theorem~9.4]{BI} there exist $(0,1)$-matrices $\overline{A}_j$ ($0\leq j\leq t$) such that 
\begin{align*}
\sum_{i\in \mathcal{I}_j}A_i=\overline{A}_j\otimes J_q,
\end{align*}
and the matrices $\overline{A}_j$ ($0\leq j\leq t$) define an association scheme on the set of fibers. 
This is called the \emph{quotient association scheme} with respect to $\mathcal{I}$.

Let $W$ be a BGW$(v,k,\lambda;\{1,-1\})$.  Write $W=W_1-W_2$ where $W_1,W_2$ are disjoint $(0,1)$-matrices. 
Let $P=\begin{bmatrix}0 & 1 \\ 1 & 0\end{bmatrix}$. 
Define the adjacency matrices as follows:
\begin{align}
A_i&=\begin{bmatrix}
P^{i-1}\otimes I_v&0\\
0& P^{i-1}\otimes I_v
\end{bmatrix}\quad (i \in\{0,1\}) ,\label{eq:a1}\\
A_{2}&=\begin{bmatrix}
J_2 \otimes (J_v-I_v)&0\\
0&J_2 \otimes (J_v-I_v)
\end{bmatrix},\label{eq:a2}\\
A_{3}&=\begin{bmatrix}
0&I_2\otimes W_1+P\otimes W_2\\
I_2\otimes W_1^\top+P\otimes W_2^\top&0
\end{bmatrix},\nonumber\\
A_{4}&=\begin{bmatrix}
0&I_2\otimes W_2+P\otimes W_1\\
I_2\otimes W_2^\top+P\otimes W_1^\top&0
\end{bmatrix},\nonumber\\
A_{5}&=\begin{bmatrix}
0&J_2\otimes(J_v-W_1-W_2)\\
J_2\otimes(J_v-W_1^\top-W_2^\top)&0
\end{bmatrix}.\nonumber
\end{align}

Note that when the weight $k$ is equal to the order $v$, namely $W$ is a Hadamard matrix, the matrix $A_{5}=0$ and this case is dealt in \cite[Section~1.8]{BCN} as Hadamard graphs.

\begin{theorem}\label{thm:as}
Let $W$ be a balanced generalized weighing matrix BGW$(v,k,\lambda)$ over $\{1,-1\}$.
If $v>k$, then $\{A_i\}_{i=0}^5$ is an association scheme with the eigenmatrices $P,Q$ given by 
\begin{align*}
P&=\left[
\begin{array}{cccccc}
 1 & 1 & 2 (v-1) & k & k & 2 (v-k) \\
 1 & -1 & 0 & \sqrt{k} & -\sqrt{k} & 0 \\
 1 & -1 & 0 & -\sqrt{k} & \sqrt{k} & 0 \\
 1 & 1 & 2 (v-1) & -k & -k & -2(v-k) \\
 1 & 1 & -2 & -\sqrt{\frac{(v-k) k}{v-1}} & -\sqrt{\frac{(v-k) k}{v-1}} & 2 \sqrt{\frac{(v-k) k}{v-1}} \\
 1 & 1 & -2 & \sqrt{\frac{(v-k) k}{v-1}} & \sqrt{\frac{(v-k) k}{v-1}} & -2 \sqrt{\frac{(v-k) k}{v-1}} \\
\end{array}
\right],\displaybreak[0]\\
Q&=\left[
\begin{array}{cccccc}
 1 & v & v & 1 & v-1 & v-1 \\
 1 & -v & -v & 1 & v-1 & v-1 \\
 1 & 0 & 0 & 1 & -1 & -1 \\
 1 & \frac{v}{\sqrt{k}} & -\frac{v}{\sqrt{k}} & -1 & -\sqrt{\frac{(v-1)(v-k)}{k}} & \sqrt{\frac{(v-1)(v-k)}{k}} \\
 1 & -\frac{v}{\sqrt{k}} & \frac{v}{\sqrt{k}} & -1 & -\sqrt{\frac{(v-1)(v-k)}{k}} & \sqrt{\frac{(v-1)(v-k)}{k}} \\
 1 & 0 & 0 & -1 & \sqrt{\frac{(v-1)k}{v-k}} & -\sqrt{\frac{(v-1)k}{v-k}} \\
\end{array}
\right].
\end{align*}
\end{theorem}
\begin{proof}
By Lemma~\ref{lem:bgw}, we have 
\begin{align*}
W_1W_1^\top+W_2W_2^\top&=W_1^\top W_1+W_2^\top W_2=\frac{1}{2}((2k-\lambda)I_v+\lambda J_v), \\
W_1W_2^\top+W_2W_1^\top&=W_1^\top W_2+W_2^\top W_1=\frac{1}{2}\lambda(J_v-I_v).
\end{align*}
It follows readily from the equations above that $A_i$'s form an association scheme. 

It is straightforward to see that the intersection numbers $B_3$ is given by 
$$B_3=\left(
\begin{array}{cccccc}
 0 & 0 & 0 & 1 & 0 & 0 \\
 0 & 0 & 0 & 0 & 1 & 0 \\
 0 & 0 & 0 & k-1 & k-1 & k \\
 k & 0 & \frac{(k-1) k}{2 (v-1)} & 0 & 0 & 0 \\
 0 & k & \frac{(k-1) k}{2 (v-1)} & 0 & 0 & 0 \\
 0 & 0 & \frac{(v-k) k}{v-1} & 0 & 0 & 0 \\
\end{array}
\right),$$
and then apply \cite[Theorem~4.1 (ii)]{BI} to this case to obtain the desired eigenmatrices. 
\end{proof}

\begin{theorem}
If there exists an association scheme with the eigenmatrices given in Theorem~\ref{thm:as}, then there exists a balanced  weighing matrix with parameters $(v,k,\lambda)$ where $\lambda=\frac{k(k-1)}{v-1}$.  
\end{theorem}
\begin{proof}
Consider the matrix $A_1+A_2$. Its eigenvalues are $2v-1,-1$ with multiplicities $2,4v-2$ respectively. Then $A_1+A_2$ is the adjacency matrix of a two copies of $K_{2v}$, the complete graph on $2v$ vertices. 
Next by the eigenvalues of $A_1$, $A_1$ is the adjacency matrix of a graph of disjoint $2v$ edges.  
Thus we may assume that $A_1,A_2$ are of the form in \eqref{eq:a1}, \eqref{eq:a2}.  
  
Write $$A_3=\begin{pmatrix}  0 & 0 & X_1 & X_2 \\
 0 & 0 & X_3 & X_4 \\
 X_1^\top & X_3^\top & 0 & 0 \\
 X_2^\top & X_4^\top & 0 & 0
\end{pmatrix},
A_4=\begin{pmatrix} 0 & 0 & Y_1 & Y_2 \\
 0 & 0 & Y_3 & Y_4 \\
 Y_1^\top & Y_3^\top & 0 & 0 \\
 Y_2^\top & Y_4^\top & 0 & 0 
\end{pmatrix}.$$ 
By the given eigenmartrices, we find that $A_1A_3=A_4$, which yields 
$$
\begin{pmatrix} 0 & 0 & X_3 & X_4 \\
 0 & 0 & X_1 & X_2 \\
 X_2^\top & X_4^\top & 0 & 0 \\
 X_1^\top & X_3^\top & 0 & 0 
\end{pmatrix}=\begin{pmatrix} 0 & 0 & Y_1 & Y_2 \\
 0 & 0 & Y_3 & Y_4 \\
 Y_1^\top & Y_3^\top & 0 & 0 \\
 Y_2^\top & Y_4^\top & 0 & 0 
\end{pmatrix}.
$$
Thus, there exist some $(0,1)$-matrices $W_1,W_2$ such that 
$$A_3=\begin{pmatrix}  0 & 0 & W_1 & W_2 \\
 0 & 0 & W_2 & W_1 \\
 W_1^\top & W_2^\top & 0 & 0 \\
 W_2^\top & W_1^\top & 0 & 0
\end{pmatrix},
A_4=\begin{pmatrix} 0 & 0 & W_2 & W_1 \\
 0 & 0 & W_1 & W_2 \\
 W_2^\top & W_1^\top & 0 & 0 \\
 W_1^\top & W_2^\top & 0 & 0 
\end{pmatrix}.$$ 
Again, by the eigenmatrices, we have that 
\begin{align*}
A_3^2&=kA_0+\frac{(k-1) k}{2 (v-1)}A_2, \\
A_3A_4&=A_4A_3=kA_1+\frac{(k-1) k}{2 (v-1)}A_2, \\
A_4^2&=kA_0+\frac{(k-1) k}{2 (v-1)}A_2.
\end{align*}
From these it follows that $(A_3-A_4)^2=2k(A_0-A_1), (A_3+A_4)^2=2k(A_0+A_1)+\frac{2(k-1) k}{v-1}A_2$, that is,  
\begin{align*}
(W_1-W_2)(W_1^\top-W_2^\top)&=kI_v,\\ 
(W_1+W_2)(W_1^\top+W_2^\top)&=(k-\lambda)I_v+\lambda J_v. 
\end{align*}
Hence, $W_1-W_2$ is a balanced generalized weighing matrix with the parameters $(v,k,\lambda)$.  
\end{proof}

\begin{remark}
Let $\{A_i\}_{i=0}^5$ be an association scheme with the same eigenmatrices as in Theorem~\ref{thm:as}.  Then it is easy to see that the quotient scheme with respect to the equivalence relation $R_0\cup R_1$, which is a three-class association scheme in \cite[Theorem~1.6.1]{BCN}, has the adjacency matrices 
\begin{align*}
\begin{bmatrix}
I_v & 0 \\ 
0 & I_v
\end{bmatrix},\begin{bmatrix}
J_v-I_v & 0 \\ 
0 & J_v-I_v
\end{bmatrix}, 
\begin{bmatrix}
0 & N \\ 
N^\top & 0
\end{bmatrix}, 
\begin{bmatrix}
0 & J_v-N \\ 
J_v-N^\top & 0
\end{bmatrix}, 
\end{align*}
where $N=W_1+W_2$ is the incidence matrix of a symmetric $(v,k,\lambda)$ design.  
\end{remark}

\section*{Acknowledgments.}
The authors would like to thank the referees for their careful reading and pointing out errors in earlier version.  
Hadi Kharaghani is supported by the Natural Sciences and
Engineering  Research Council of Canada (NSERC).  Sho Suda is supported by JSPS KAKENHI Grant Number 18K03395.


\begin{thebibliography}{99}
\bibitem{BI}
E. Bannai, T. Ito, Algebraic Combinatorics I: Association Schemes,
{Benjamin/Cummings, Menlo Park, CA,} 1984.

\bibitem{BCN}
A. E.~Brouwer, A.~E.~Cohen and A.~Neumaier, 
Distance-regular graphs, 
Springer-Verlag, Berlin, 1989. xviii+495.

\bibitem{D}
P. Delsarte. An algebraic approach to the association schemes of coding theory. {\it Philips Res.
Rep. Suppl.}, (10):vi+97, 1973.

\bibitem{HSS}
A. S. Hedayat, N. J. A. Sloane, and J. Stufken, {\it Orthogonal Arrays: Theory and Applications.} New York: Springer-Verlag, 1999.

\bibitem{ds83}
W. de Launey, and D. G. Sarvate, Non-existence of certain GBRDs, {\it Ars Combin.}, {\bf 18}, (1983), 5--20.

\bibitem{GM} P. B. Gibbons, R. Mathon, Construction methods for Bhaskar Rao and related designs, {\it J. Aust. Math. Soc., Ser. A} {\bf 42},  (1987), 5–30.

\end{thebibliography}
\end{document}